\documentclass[reqno, 12pt]{amsart}
\usepackage{amssymb,amsmath,amsfonts,amsthm,comment,mathrsfs,times,graphicx}
\usepackage[bookmarksnumbered, colorlinks, plainpages]{hyperref}
\usepackage{color}
\usepackage[english]{babel}
\usepackage[all,cmtip]{xy}
\usepackage{lmodern}
\usepackage{enumitem}
\usepackage{geometry}
\newcounter{dummy} \numberwithin{dummy}{section}
\newtheorem{theo}[dummy]{Theorem }
 \newtheorem{coro}[dummy]{Corollary}
 \newtheorem{lem}[dummy]{Lemma}
 \newtheorem{pro}[dummy]{Proposition}
\newtheorem{deft}[dummy]{Definition}
\newtheorem{exe}[dummy]{Example}
\newtheorem{rem}[dummy]{Remark}

\title[On Iwasawa theory of Rubin-Stark units and narrow class groups ]{On Iwasawa theory of
 Rubin-Stark units and narrow class groups}
\author[ Youness Mazigh]{ Youness Mazigh }
\address{ Université Moulay Ismaïl,
 Département de mathématiques,
Faculté des sciences de Meknès, B.P. 11201 Zitoune, Meknès, Maroc.}
\email{\textcolor[rgb]{0.00,0.00,1.00}{y.mazigh@edu.umi.ac.ma}}
\begin{document}
\renewcommand{\abstractname}{Abstract}
\begin{abstract}
Let $K$ be a totally real number field of degree $r$.  Let
$K_{\infty}$ denote the cyclotomic $\mathbb{Z}_{2}$-extension of $K$
and let $L_{\infty}$ be a finite extension of $K_{\infty}$, abelian
over $K$.  The goal of this paper is to compare the characteristic
ideal of the $\chi$-quotient of the projective limit of the narrow
class groups to the $\chi$-quotient of the projective limit of the
$r$-th exterior power of totally positive units modulo a subgroup of
Rubin-Stark units, for some $\overline{\mathbb{Q}_{2}}$-irreducible
characters $\chi$ of $\mathrm{Gal}(L_{\infty}/K_{\infty})$.
\end{abstract}
\maketitle
\section{Introduction}
 Let $K$ be a  number field and let $Cl_{K}^{+}$ denote the narrow class group of $K$, that
is, the quotient group of the group of fractional ideals of $K$
modulo the subgroup of principal fractional ideals generated by a
 totally positive element $\alpha$ of $K$, $\mathrm{i.e.}$ $\alpha$ is an element
of $K^{\ast}$ such that $\sigma(\alpha)$ is positive for every
embedding $\xymatrix@=1.5pc{\sigma :\, K \ar[r]& \mathbb{R}}$. The
natural homomorphism of $Cl^{+}_{K}$ onto the ideal class group of
$K$ induces, for every odd prime $p$, an isomorphism of the
$p$-primary component of $Cl_{K}^{+}$ onto the $p$-class group of
$K$. But the $2$-primary components are not necessarily  isomorphic.
Before we explain our results in  details, we set some notation.
\vskip 6pt Let $K$ be a totally real number field of degree
$r=[K:\mathbb{Q}]$. Let $K_{\infty}$ denote the cyclotomic
$\mathbb{Z}_{2}$-extension of $K$ and $L_{\infty}$ a finite
extension of $K_{\infty}$, abelian over $K$. Fix a decomposition of
\begin{equation*}
\mathrm{Gal}(L_{\infty}/K)=\mathrm{Gal}(L_{\infty}/K_{\infty})\times
\Gamma,\; \Gamma\simeq \mathbb{Z}_{2}.
\end{equation*}
Then the fields $L:=L_{\infty}^{\Gamma}$ and $K_{\infty}$ are
linearly disjoint over $K$. \vskip 7pt If $F/K$ is a finite abelian
extension of $K$, we write $A^{+}(F)$ for the $2$-part of the narraw
class group of $F$ and $\mathcal{E}^{+}(F)$ for the group of the
totally positive units of $F$. For a $\mathbb{Z}$-module $M$, let
$\widehat{M}=\varprojlim M/2^{n}M$ denote the $2$-adic completion of
$M$. Let
\begin{equation*}
A^{+}_{\infty}:=\varprojlim A^{+}(F)\quad\quad\mbox{and}\quad\quad
\widehat{\mathcal{E}^{+}_{\infty}}:=\varprojlim
\widehat{\mathcal{E}^{+}(F)},
\end{equation*}
where the projective limit is taken over all finite sub-extensions
of $L_{\infty}$, with respect to  the norm maps. Let
\begin{equation*}
\xymatrix@=2pc{ \chi : G_{K}\ar[r]&
\overline{\mathbb{Q}}_{2}^{\times}}
\end{equation*}
be a non-trivial totally even character of the absolute Galois
$G_{K}$ of $K$ ($\mathrm{i.e.}$, it is trivial on all complex
conjugations inside $G_{K}$) that  factors through $L$. Denote the
ring generated by the values of $\chi$ over $\mathbb{Z}_{2}$  by
$\mathcal{O}$, and let $\Delta$ be the Galois group
$\mathrm{Gal}(L/K)$. Let $\mathcal{O}(\chi)$ denote the ring
$\mathcal{O}$ on which $\Delta$ acts via $\chi$. For any
$\mathbb{Z}_{2}[\Delta]$-module $M$, we define the $\chi$-quotient
$M_{\chi}$ of $M$ by:
\begin{equation*}
M_{\chi}:=M\otimes_{\mathbb{Z}_{2}[\Delta]}\mathcal{O}(\chi).
\end{equation*}
For any profinite group $\mathcal{G}$, we define the Iwasawa algebra
\begin{equation*}
\mathcal{O}[[\mathcal{G}]]:=\varprojlim
\mathcal{O}[\mathcal{G}/\mathcal{H}]
\end{equation*}
where the projective limit is over all finite quotient
$\mathcal{G}/\mathcal{H}$ of $\mathcal{G}$. In case
$\mathcal{G}=\Gamma$, we shall write
\begin{equation*}
\Lambda:=\mathcal{O}[[\Gamma]].
\end{equation*}
 Let $L_{\chi}$ denote the fixed field of $\ker(\chi)$ and let $K(1)$ be the
maximal $2$-extension inside the Hilbert class field of $K$. In the
sequel we will assume  (for simplicity) that
\begin{equation*}\label{simplicity}
L=L_{\chi} \quad\mbox{and}\quad K=L\cap K(1).
\end{equation*}
In particular, $L$ is totally real.\vskip 6pt
 For a $2$-adic prime $\mathfrak{p}$ of $K$, let
 $\mathrm{Frob}_{\mathfrak{p}}$ denote a Frobenius element at
 $\mathfrak{p}$ inside the absolute Galois group of $K$.
 Assume that
\begin{center}
  \begin{tabbing}
  \hspace{0.5cm}\=\hspace{0.4cm} \= \hspace{0.4cm} \= \kill
  \>$(\mathcal{H}_{1})$ \> \> \textit{ the extension $L/\mathbb{Q}$ is unramified at $2$,}
\end{tabbing}
\end{center}
\begin{center}
  \begin{tabbing}
  \hspace{0.5cm}\=\hspace{0.4cm} \= \hspace{0.4cm} \= \kill
  \>$(\mathcal{H}_{2})$ \> \> \textit{  for any $2$-adic prime $\mathfrak{p}$ of $K$,
  we have $\chi(\mathrm{Frob}_{\mathfrak{p}})\neq
    1$, }
\end{tabbing}
\end{center}
\begin{center}
  \begin{tabbing}
  \hspace{0.5cm}\=\hspace{0.4cm} \= \hspace{0.4cm} \= \kill
  \>$(\mathcal{H}_{3})$ \> \> \textit{
   the Leopoldt  conjecture holds for every finite extension $F$ of $L$ in $L_{\infty}$ for the prime $2$.}
\end{tabbing}
\end{center}
 We will denote by
$\widehat{\mathrm{St}^{+}_{\infty}}$ the projective limit
$\displaystyle{\varprojlim_{n}}\widehat{St_{n}^{+}}$, where
$St_{n}^{+}$ is the module constructed by the Rubin-Stark elements
(see Definition \ref{Definition of Rubin Strak module}). In
particular $\widehat{\mathrm{St}^{+}_{\infty}}$ is a submodule of
$\displaystyle{\bigwedge^{r}}\widehat{\mathcal{E}^{+}_{\infty}}$.\\
In \cite{Ma 16}, for a fixed odd rational prime $p$, we used the
theory of Euler systems to bound the size of the $\chi$-quotient of
the $p$-class groups by the characteristic ideal of the
$\chi$-quotient of the $r$-th exterior power of units modulo
Rubin-Stark units, in the non semi-simple case, extending thus
results
of \cite{Kazim109}. \\
In this paper, we consider the case $p=2$. More precisely, we use
the Euler system formed by Rubin-Stark elements to compare the
characteristic ideal of the $\chi$-quotient of the projective limit
of the $2$-part of the narrow class groups to the $\chi$-quotient of
the projective limit of the $r$-th exterior power of totally
positive units modulo $\widehat{\mathrm{St}^{+}_{\infty}}$. We draw
the attention of the reader to the fact that, because of many
complications, the case $p=2$ is not often treated in the
literature, unlike \cite{Greither92, Oukhaba 12}. The following
theorem summarizes our results.
\begin{theo}\label{Mazigh 2}
 Assume that the hypotheses $(\mathcal{H}_{1}),(\mathcal{H}_{2}) \,\mbox{and}\, (\mathcal{H}_{3})$
 hold. Then
\begin{equation*}
\mathrm{char}((A^{+}_{\infty})_{\chi})\quad \mbox{divides}\quad
\lambda. \mathrm{char}
\big(((\displaystyle{\bigwedge^{r}}\widehat{\mathcal{E}^{+}_{\infty}})/\widehat{\mathrm{St}^{+}_{\infty}})_{\chi}\big)
\end{equation*}
where $\lambda$ is a power of $2$ explicitly given in formula
$(\ref{lamda factur})$.
\end{theo}
Treating the case  $p=2$ leads to several complications. The first
comes from the non triviality of the cohomology groups of the
absolute Galois group of $\mathbb{R}$. More precisely, for a number
field $F$ and a real place $w$ of $F$, the cohomology group
$H^{i}(F_{w},M)$ is not necessarily trivial, where $M$ is a
$\mathbb{Z}_{p}[G_{F_{w}}]$-module. Hence the result of
\cite[Proposition 3.8]{AMO1} do not apply, since the cohomological
dimension of $G_{K,\Sigma}$ is infinite. The second complication is
the need to modify the canonical Selmer structure
$\mathcal{F}_{can}$, and to study the $\Lambda$-structure of the
projective limit of these Selmer groups. This problem is treated in
subsection $\ref{section Iwasawa theory}$. For this we use a
relation between the universal norms in $\mathbb{Z}_{p}$-extension
and the $\Lambda$-structure of certain modules. This is already
known, thanks to Vauclair who applied some homological proprieties
in \cite{Va 06, Va 09} to determine this relation.
 \vskip 6pt To control the
contributions from infinite places, we use a slight variant of
Galois cohomology, the so-called totally positive Galois cohomology
$H^{\ast}_{+}(G_{K,\Sigma},.)$, see subsection \ref{subsection
totally positive Galois cohomology}, introduced by Kahn in \cite{Ka
93}, based on ideas of Milne \cite{Milne}. Totally positive Galois
cohomology has been used by several authors, such as \cite{CKPS} and
\cite{Assim mova}.\vskip6pt
 \noindent\textbf{ Acknowledgements.} I am very grateful to J. Assim and H. Oukhaba for
making several helpful comments and for their careful reading of the
preliminary versions.
\section{\bf Iwasawa theory of Selmer groups}\label{section Selmer structures}
\subsection{Selmer structures}
In this subsection we recall  some definitions concerning the notion
of Selmer structure introduced by Mazur and Rubin in \cite{MR04,MR
16}.
 For any field $k$  and a fixed separable algebraic closure
$\overline{k}$ of $k$, we write
$G_{k}:=\mathrm{Gal}(\overline{k}/k)$ for the Galois group of
$\overline{k}/k$. Let $\mathcal{O}$ be the ring of integers of a
finite extension $\Phi$ of $\mathbb{Q}_{2}$ and  let $D$ denote the
divisible module $\Phi/\mathcal{O}$. For a $2$-adic representation
$T$ with coefficients in $\mathcal{O}$, we  define
\begin{equation*}
 D(1)=D\otimes \mathbb{Z}_{2}(1),\quad\quad
T^{\ast}=\mathrm{Hom}_{\mathcal{O}}(T,D(1)),
\end{equation*}
where $\mathbb{Z}_{2}(1):=\varprojlim \mu_{2^{n}}$ is the Tate
module.\vskip 6pt For a number field $F$, let $F_{w}$ denote the
completion of $F$ at a given place $w$ of $F$. Let us recall the
local duality theorem $\mathrm{cf}.$\,\cite[Corollary I.2.3
]{Milne}: For $i=0,1,2$, there is a perfect  pairing
\begin{equation}\label{local dualty theorem}
\begin{array}{cccccccc}
  H^{2-i}(F_{w},T) & \times & H^{i}(F_{w},T^{\ast})& \xymatrix@=3pc{\ar[r]^-{\langle\;,\;
  \rangle_{w}}&} &H^{2}(F_{w},D(1))\cong D,\mbox{if $w$ is finite},\\
 &  &   &  &\\
  \widehat{H}^{2-i} (F_{w},T)& \times & \widehat{H}^{i}(F_{w},T^{\ast}) & \xymatrix@=3pc{\ar[r]^-{\langle\;,\;\rangle_{w}}&} &
\widehat{H}^{2}(F_{w},D(1)),\mbox{if $w$ is infinite}
\end{array}
\end{equation}
where $\widehat{H}^{\ast}(F_{w},.)$ denotes the Tate cohomology
group.\vskip 6pt
\begin{deft} Let $T$ be a $2$-adic representation of $G_{F}$ with
coefficients in $\mathcal{O}$ and let $w$ be a non $2$-adic prime of
$F$. A local condition $\mathcal{F}$ at the prime $w$  on $T$ is a
choice of an $\mathcal{O}$-submodule $H^{1}_{\mathcal{F}}(F_{w},T)$
of $H^{1}(F_{w},T)$. For the $2$-adic primes, a local condition at
$2$ will be a choice of an $\mathcal{O}$-submodule
$H^{1}_{\mathcal{F}}(F_{2},T)$ of the semi-local cohomology group
\begin{equation*}
H^{1}(F_{2},T):=\oplus_{w\mid 2}H^{1}(F_{w},T).
\end{equation*}
\end{deft}
Let $I_{w}$ denote the inertia subgroup of $G_{F_{w}}$. We say that
$T$ is unramified at $w$ if the inertia subgroup $I_{w}$ of $w$ acts
trivially on $T$. We assume in the sequel that $T$ is unramified
outside a finite set of places of $F$.
\begin{deft}
A Selmer structure $\mathcal{F}$ on $T$ is a collection of the
following data:
\begin{itemize}
    \item a finite set $\Sigma(\mathcal{F})$ of places of $F$, including all infinite places, all $2$-adic places and all
    places where $T$ is ramified,
    \item  a local condition on
    $T$, for every $v\in \Sigma(\mathcal{F})$.
\end{itemize}
If $w\not\in \Sigma(\mathcal{F})$ we will write
$H^{1}_{\mathcal{F}}(F_{w},T)=H^{1}_{ur}(F_{w},T)$, where
$H^{1}_{ur}(F_{w},T)$ is the subgroup of unramified cohomology
classes:
\begin{equation*}
H^{1}_{ur}(F_{w},T)=\ker(\xymatrix@=2pc{H^{1}(F_{w},T)\ar[r]&
H^{1}(I_{w},T)}).
\end{equation*}
\end{deft}
 If $\mathcal{F}$ is a Selmer structure on $T$, we define the Selmer
 group $H^{1}_{\mathcal{F}}(F,T)\subset H^{1}(F,T)$ to be the kernel of the localization map
\begin{equation*}\label{Selmer group definition}
\xymatrix@=2pc{H^{1}(G_{F,\Sigma(\mathcal{F})},T)\ar[r]&
\displaystyle{\bigoplus_{w\in\Sigma(\mathcal{F})}}(H^{1}(F_{w},T)/H^{1}_{\mathcal{F}}(F_{w},T))},
\end{equation*}
where
$G_{F,\Sigma(\mathcal{F})}:=\mathrm{Gal}(F_{\Sigma(\mathcal{F})}/F)$
is the Galois group of the maximal algebraic extension of $F$
unramified outside $\Sigma(\mathcal{F})$.\vskip 6pt A  Selmer
structure $\mathcal{F}$ on $T$ determines a Selmer structure
$\mathcal{F}^{\ast}$ on $T^{\ast}$. Namely,
\begin{equation*}
\Sigma(\mathcal{F})=\Sigma(\mathcal{F}^{\ast}),\quad
H^{1}_{\mathcal{F}^{\ast}}(F_{w},T^{\ast}):=H^{1}_{\mathcal{F}}(F_{w},T)^{\perp},\;\;
\mbox{if $w\in\Sigma(\mathcal{F}^{\ast})-\Sigma_{2}$}
\end{equation*}
 under the local Tate pairing $\langle\;,\;\rangle_{w}$
and
\begin{equation*}
H^{1}_{\mathcal{F}^{\ast}}(F_{2},T^{\ast}):=H^{1}_{\mathcal{F}}(F_{2},T)^{\perp},
\end{equation*}
 under the pairing $\oplus_{w\mid 2}\langle\;,\;\rangle_{w}$. Here
 $\Sigma_{2}$ denotes the set of $2$-adic places of $F$. \vskip 6pt
There is a natural partial ordering on the set of Selmer structures
on $T$. Namely, we will say that $\mathcal{F}\leq
\mathcal{F}^{\prime}$ if and only if
\begin{equation*}
H^{1}_{\mathcal{F}}(F_{w},T)\subset
H^{1}_{\mathcal{F}^{\prime}}(F_{w},T)\;\;\mbox{for all places $w$.}
\end{equation*}
If  $\mathcal{F}\leq \mathcal{F}^{\prime}$ we have an exact sequence
\cite[Theorem 2.3.4]{MR04}
\begin{equation}\label{exacte sequence of selmer structures}
\xymatrix@=1pc{H^{1}_{\mathcal{F}}(F,T)\ar@{^{(}->}[r]&H^{1}_{\mathcal{F}^{\prime}}(F,T)\ar[r]&
\bigoplus_{w}H^{1}_{\mathcal{F}^{\prime}}(F_{w},T)/H^{1}_{\mathcal{F}}(F_{w},T)\ar[r]&
H^{1}_{\mathcal{F}^{\ast}}(F,T^{\ast})^{\vee}\ar@{->>}[r]&
H^{1}_{\mathcal{F}^{\prime,\ast}}(F,T^{\ast})^{\vee}}
\end{equation}
\begin{exe}\label{Example unramified local condition}
 Let $w$ be a place of $F$ and  let $F_{w}^{ur}$ denote the maximal
 unramified extension of $F_{w}$. Define the subgroup of universal
 norms
\begin{equation*}
H^{1}(F_{w},T)^{u}=\bigcap_{F_{w}\subset k\subset
F_{w}^{ur}}\mathrm{cor}_{k,F_{w}}H^{1}(k,T),
\end{equation*}
 where the intersection is over all finite unramified extensions $k$ of $F_{w}$.
Let $H^{1}(F_{w},T)^{u,sat}$ denote the $\mathcal{O}$-saturation of
$H^{1}(F_{w},T)^{u}$ in $H^{1}(F_{w},T)$, $\mathrm{ i.e.}$,
$H^{1}(F_{w},T)/H^{1}_{\mathcal{F}_{ur}}(F_{w},T)$ is a free
$\mathcal{O}$-module and
$H^{1}_{\mathcal{F}_{ur}}(F_{w},T)/H^{1}(F_{w},T)^{u}$ has finite
length. Following \cite[Defintition 5.1]{MR 16}, we define the
unramified Selmer structure $\mathcal{F}_{ur}$ on $T$ by\vskip 6pt
\begin{itemize}
    \item $\Sigma(\mathcal{F}_{ur}):=\{\mathfrak{q}:\, T\,\mbox{is ramified
    at}\,\mathfrak{q}\}\cup\{\mathfrak{p}:\,\mathfrak{p}\mid 2\}\cup
    \{ v:\, v\mid\infty\}$
 \item
 $H^{1}_{\mathcal{F}_{ur}}(F_{w},T)=\left\{
                                               \begin{array}{ll}
                                                H^{1}(F_{w},T)^{u,sat} , & \hbox{if $w\nmid p\infty$;} \\
                                                H^{1}(F_{w},T) , & \hbox{if $w\mid \infty$.}
                                               \end{array}
                                             \right.
                                             $, and\;
                                             $H^{1}_{\mathcal{F}_{ur}}(F_{2},T)=\bigoplus_{\mathfrak{p}\mid
2}H^{1}(F_{\mathfrak{p}},T)^{u,sat}$.
\end{itemize}
For future use, we record here the following well-known properties
of unramified Selmer structure
\begin{enumerate}[label=(\roman*)]
    \item
    \begin{equation}\label{unramified properties}
    H^{1}_{\mathcal{F}_{ur}^{\ast}}(F_{w},T^{\ast})=H^{1}_{ur}(F_{w},T^{\ast})_{div},\quad
H^{1}_{\mathcal{F}_{ur}^{\ast}}(F_{2},T^{\ast})=\bigoplus_{\mathfrak{p}\mid
2}H^{1}_{ur}(F_{\mathfrak{p}},T^{\ast})_{div}.
\end{equation}
    \item If $w\nmid 2$ and $T$ is
unramified at $w$, then
\begin{equation*}
H^{1}_{\mathcal{F}_{ur}}(F_{w},T)=H^{1}_{ur}(F_{w},T)\quad\mbox{and}\quad
H^{1}_{\mathcal{F}_{ur}^{\ast}}(F_{w},T^{\ast})=H^{1}_{ur}(F_{w},T^{\ast}).
\end{equation*}
\item Let $Cl_{F}$ denote the ideal class group of $F$. Then
\begin{equation*}\label{Structure and class group}
H^{1}_{\mathcal{F}_{ur}^{\ast}}(F,\mathbb{Q}_{2}/\mathbb{Z}_{2})^{\vee}\cong
Cl_{F}\otimes \mathbb{Z}_{2}.
\end{equation*}
\end{enumerate}
 where for an
abelian group $A$, $A_{div}$ denotes the maximal divisible subgroup
of $A$, and $()^{\vee}$ denotes the Pontryagin dual.\\
The assertion $(\textrm{i})$ follows from \cite[\S 2.1.1,
Lemme]{PR92} and the assertions $(\textrm{ii})$ and $(\textrm{iii})$
follow immediately from  \cite[Lemma 1.3.5]{Rubin00} and \cite[\S
6.1]{MR04} respectively.
\end{exe}
\subsection{Totally
positive Galois cohomology}\label{subsection totally positive Galois
cohomology} Let $\Sigma$ be a finite set of places of $F$ containing
infinite places and all  $2$-adic places. If $F^{\prime}$ is an
extension of $F$, we denote also by $\Sigma$ the set of places of
$F^{\prime}$ lying above places in $\Sigma$. Let $G_{F,\Sigma}$ be
the Galois group of the maximal algebraic extension $F_{\Sigma}$ of
$F$ which is unramified outside $\Sigma$. If $w$ is a place of $F$,
we denote the decomposition group of $w$
in $\overline{F}/F$ by $G_{w}$.\\
For a finite $\mathcal{O}[G_{F,\Sigma}]$-module $M$, we write
$M_{+}$ for the the cokernel of the injective map $\xymatrix@=2pc{
M\ar[r]& \displaystyle{\oplus_{w\mid
\infty}\mathrm{Ind}^{G_{F}}_{G_{w}}}M}$;
\begin{equation*}
\xymatrix@=2pc{0\ar[r]&M\ar[r]&\displaystyle{\oplus_{w\mid
\infty}\mathrm{Ind}^{G_{F}}_{G_{w}}}M\ar[r]&M_{+}\ar[r]&0},
\end{equation*}
 where
$\mathrm{Ind}^{G_{K}}_{G_{w}}M$ denotes the induced module.
Following \cite{Ka 93}, we define the $i$-$\mathrm{th}$ totally
positive Galois cohomology $H^{i}_{+}(G_{F,\Sigma},M)$ of $M$ by
\begin{equation*}
H^{i}_{+}(G_{F,\Sigma},M):=H^{i-1}(G_{F,\Sigma},M_{+}).
\end{equation*}
We first list the following facts which hold for an arbitrary number
field $F$.
\begin{pro}\label{Proposition properties of positive cohomology}
 We have the following properties:
\begin{enumerate}[label=(\roman*)]
    \item There is a long exact sequence
    \begin{equation*}
    \xymatrix@=1pc{\cdots\ar[r]&H^{i}_{+}(G_{F,\Sigma},M)\ar[r]&
    H^{i}(G_{F,\Sigma},M)\ar[r]&
    \displaystyle{\oplus_{w\mid\infty}}H^{i}(F_{w},M)\ar[r]&H^{i+1}_{+}(G_{F,\Sigma},M)\ar[r]&\cdots}.
    \end{equation*}
    \item For $i\not\in\{1,2\}$, we have
    $H^{i}_{+}(G_{F,\Sigma},M)=0$.
    \item If $F^{\prime}/F$ is an extension unramified outside
    $\Sigma$ with Galois group
    $G$ then there is a cohomological spectral sequence
    \begin{equation*}
    H^{p}(G,H^{q}_{+}(G_{F^{\prime},\Sigma},M))\Longrightarrow
    H^{p+q}_{+}(G_{F,\Sigma},M).
    \end{equation*}
\end{enumerate}
\end{pro}
\begin{proof}
See \cite[\S 5]{Ka 93}.
\end{proof}
The following corollary is a  direct consequence of $(\textit{ii})$
in Proposition \ref{Proposition properties of positive cohomology}
above.
\begin{coro}\label{corolary coresriction}
Let $F^{\prime}/F$ be a $\Sigma$-ramified extension with Galois
group $G$. Then the corestriction  map
\begin{equation*}
\xymatrix@=2pc{ H^{2}_{+}(G_{F^{\prime},\Sigma},M)_{G}\ar[r]&
H^{2}_{+}(G_{F,\Sigma},M)}
\end{equation*}
is an isomorphism.
\end{coro}
To go further, we need the following remark. If  $M_{\Sigma}$
denotes the cokernel of the canonical map $\xymatrix@=2pc{ M\ar[r]&
\oplus_{w\in \Sigma}\mathrm{Ind}^{G_{F}}_{G_{w}} M}$, then for all
$i\geq 0$, we have
\begin{equation}\label{cohomology of MS}
H^{i}(G_{F,\Sigma},M_{\Sigma})=H^{i+1}_{c}(G_{F,\Sigma},M),
\end{equation}
where $H^{i+1}_{c}(G_{F,\Sigma},.)$ is the continuous  cohomology
with compact support (for the definition see \cite[ \S
5.7.2]{Nekovar06}). Note that
\begin{equation}\label{Nekovar duality}
H^{i}_{c}(G_{F,\Sigma},M)\cong
H^{3-i}(G_{F,\Sigma},M^{\ast})^{\vee},
\end{equation}
\cite[Proposition 5.7.4]{Nekovar06}, where
$M^{\ast}=\mathrm{Hom}_{\mathbb{Z}_{2}}(M,\mu_{2^{\infty}})$.
\begin{pro}\label{exact sequence of finie places}
Let $\Sigma_{f}$ denote the set of finite places in $\Sigma$. Then
there is a long exact sequence
\begin{equation}
\xymatrix@=1pc{
\displaystyle{\oplus_{w\in\Sigma_{f}}}H^{1}(F_{w},M)\ar[r]&
H^{1}(G_{F,\Sigma},M^{\ast})^{\vee}\ar[r]&
H^{2}_{+}(G_{F,\Sigma},M)\ar[r]&\displaystyle{\oplus_{w\in\Sigma_{f}}}H^{2}(F_{w},M)\ar@{->>}
[r]& H^{0}(G_{F,\Sigma},M^{\ast})^{\vee}}.
\end{equation}
\end{pro}
\begin{proof}
Consider the commutative exact diagram
\begin{equation*}
\xymatrix@=1.5pc{ 0\ar[r]&M\ar@{=}[d]\ar[r]&
\oplus_{w\in\Sigma}\mathrm{Ind}_{G_{w}}^{G_{F}}M\ar[r]\ar@{->>}[d]&
M_{\Sigma}\ar[d]\ar[r]&0\\
0\ar[r]&M\ar[r]&
\oplus_{w\mid\infty}\mathrm{Ind}_{G_{w}}^{G_{F}}M\ar[r]&M_{+}\ar[r]&0}.
\end{equation*}
Using the snake lemma we obtain the exact sequence
\begin{equation}\label{equation of exact sequence }
\xymatrix@=2pc{
0\ar[r]&\displaystyle{\oplus_{w\in\Sigma_{f}}\mathrm{Ind}^{G_{K}}_{G_{w}}}M\ar[r]&M_{\Sigma}\ar[r]
&M_{+}\ar[r]&0}.
\end{equation}
Taking the $G_{F,\Sigma}$-cohomology of the exact sequence
$(\ref{equation of exact sequence })$ and since
$H^{i}_{+}(G_{F,\Sigma},M)=0$ for $i\not\in\{1,2\}$ (see Proposition
\ref{Proposition properties of positive cohomology}), we get  the
exact sequence
\begin{equation*}
 \xymatrix@=1pc{
\displaystyle{\oplus_{w\in\Sigma_{f}}}H^{1}(K_{w},M)\ar[r]&
H^{1}(G_{K,\Sigma},M_{\Sigma})\ar[r]&
H^{2}_{+}(G_{F,\Sigma},M)\ar[r]&\displaystyle{\bigoplus_{w\in\Sigma_{f}}}H^{2}(K_{w},M)\ar@{->>}
[r]& H^{2}(G_{F,\Sigma},M_{\Sigma})}.
\end{equation*}
To obtain the desired result, it suffices to observe that
 \begin{equation*}
 H^{1}(G_{F,\Sigma},M_{\Sigma})=H^{1}(G_{F,\Sigma},M^{\ast})^{\vee}\quad\mbox{and}\quad
 H^{2}(G_{F,\Sigma},M_{\Sigma})=H^{0}(G_{F,\Sigma},M^{\ast})^{\vee};
 \end{equation*}
 this is a consequence of the  properties $(\ref{cohomology of MS})$ and $(\ref{Nekovar
duality})$.
\end{proof}
\subsection{Iwasawa theory}\label{section Iwasawa theory} Throughout this
subsection we fix a totally real number field $K$. Let
$r=[K:\mathbb{Q}]$ and $K_{\infty}=\bigcup_{n\geq 0}K_{n}$ denote
the cyclotomic $\mathbb{Z}_{2}$-extension of $K$. Assume that all
algebraic extensions of $K$ are contained in a fixed  algebraic
closure $\overline{\mathbb{Q}}$ of $\mathbb{Q}$. If $F$ is a finite
extension of $K$ and $w$ is a place of $F$, fix a place
$\overline{w}$ of $\overline{\mathbb{Q}}$ lying above $w$. The
decomposition (resp. inertia) group of $\overline{w}$ in
$\overline{\mathbb{Q}}/F$ is denoted  by $G_{w}$ (resp. $I_{w}$). If
$v$ is a place of $K$ and $F$ is a Galois extension of $K$, we
denote the decomposition group of $v$ in $F/K$ by $D_{v}(F/K)$.
Recall that
\begin{equation*}
\xymatrix@=2pc{ \chi : G_{K}\ar[r]& \mathcal{O}^{\times}}
\end{equation*}
is a non-trivial totally even  character, factoring through a finite
abelian  extension $L$ of $K$. Assume that $L$ and $K_{\infty}$ are
linearly disjoint over $K$. Let $L_{n}=LK_{n}$ and let
$L_{\infty}=LK_{\infty}$ be the cyclotomic
$\mathbb{Z}_{2}$-extension of $L$. In the sequel, we will denote by
$T$ the $2$-adic representation
\begin{equation*}
T=\mathbb{Z}_{2}(1)\otimes \mathcal{O}(\chi^{-1}).
\end{equation*}
Let $\Sigma$ be a finite set of places of $K$ containing all
infinite places, all $2$-adic places and all places where $T$ is
ramified. If $F$ is an extension of $K$, we denote also by $\Sigma$
the set of places of $F$ lying above places in $\Sigma$.\vskip 6pt
 Let's recall the definition of the canonical Selmer
structure $\mathcal{F}_{can}$ on $T$;
\begin{itemize}
    \item $\Sigma(\mathcal{F}_{can})=\Sigma$.
    \item $H^{1}_{\mathcal{F}_{can}}(F_{w},T)=\left\{
                                                \begin{array}{ll}
                                                  H^{1}_{\mathcal{F}_{ur}}(F_{w},T), & \hbox{if $w\nmid 2\infty$;} \\
                                                  H^{1}(F_{w},T), & \hbox{if $w\mid\infty$.}
                                                \end{array}
                                              \right.
                                              $\quad and
                                              $H^{1}_{\mathcal{F}_{can}}(F_{2},T)=\displaystyle{\oplus_{w\mid
                                              2}}H^{1}(F_{w},T)$,
\end{itemize}
where $\mathcal{F}_{ur}$ is the unramified local condition, see
Example \ref{Example unramified local condition}.
  Let
\begin{equation*}
H^{1}_{\mathcal{F}_{can}}(FK_{\infty},T):=\varprojlim_{n}H^{1}_{\mathcal{F}_{can}}(FK_{n},T),\quad
H^{1}_{\mathcal{F}_{can}^{\ast}}(FK_{\infty},T^{\ast}):=\varinjlim_{n}H^{1}_{\mathcal{F}_{can}^{\ast}}(FK_{n},T^{\ast}),
\end{equation*}
 where the projective (resp. injective) limit is taken with respect to
 the corestriction (resp. restriction) maps. For an $\mathcal{O}$-module $M$, we
 denote by
$M^{\vee}:=\mathrm{Hom}_{\mathbb{Z}_{2}}(M,\mathbb{Q}_{2}/\mathbb{Z}_{2})$
its Pontryagin dual.\\
 Note that the Kolyvagin system  (see \cite{Rubin00, MR04})
 machinery permits to obtain bounds on the associated Selmer groups.
 More precisely, the Kolyvagin-Rubin approach shows (see
 \cite[Theorem 2.3.3]{Rubin00}) that if a non-trivial Euler system exists,
 then the index of the Euler system in $H^{1}_{\mathcal{F}_{can}}(K_{\infty},T)$ gives a bound for
 $H^{1}_{\mathcal{F}_{can}^{\ast}}(K_{\infty},T^{\ast})^{\vee}$.
  It is  well know that the Rubin-Stark elements give rise to Euler systems for the $2$-adic
  representation $T=\mathbb{Z}_{2}(1)\otimes \mathcal{O}(\chi^{-1})$
  ($\mathrm{e.g.}$\cite{Rubin96}). To find a bound for the narrow
  class group, we need to modify the canonical Selmer structure,
  $\mathrm{cf}$. Proposition \ref{chi quotient of restrincted class
  group}.
\begin{deft}
Let $F$ be a finite extension of $K$, and let $\mathcal{F}$ be a
Selmer structure on $T$. We define the positive  Selmer structure
$\mathcal{F}^{+}$ by
\begin{itemize}
    \item $\Sigma(\mathcal{F}^{+})=\Sigma(\mathcal{F})$
    \item $H^{1}_{\mathcal{F}^{+}}(F_{w},T)=\left\{
                                                    \begin{array}{ll}
                                                      H^{1}_{\mathcal{F}}(F_{w},T), & \hbox{if $w\nmid \infty$;} \\
                                                      0, & \hbox{ if $w\mid \infty$.}
                                                    \end{array}
                                                  \right.
                                                  $
\end{itemize}
\end{deft}
The following lemma is a first step towards  our purpose.
\begin{lem}\label{Lemma narrow class}
Let $F$ be a finite Galois extension of $L$, and let $Cl^{+}_{F}$ be
the narrow class group of $F$. Then
\begin{equation*}
H^{1}_{\mathcal{F}_{ur}^{+,\ast}}(F,T^{\ast})\cong
\mathrm{Hom}(Cl^{+}_{F},T^{\ast}).
\end{equation*}
\end{lem}
\begin{proof}
Let $w$ be a finite place of $F$. Since $\chi$ is a character
factoring through $L$, the decomposition group $G_{w}$ acts
trivially on $T^{\ast}$. Then
\begin{equation*}
H^{1}_{ur}(F_{w},T^{\ast})\cong \mathrm{Hom}(G_{w}/I_{w},T^{\ast}).
\end{equation*}
Moreover $G_{w}/I_{w}$ is torsion-free and $T^{\ast}$ is divisible,
then $\mathrm{Hom}(G_{w}/I_{w},T^{\ast})$ is divisible, therefore
\begin{equation*}
H^{1}_{\mathcal{F}_{ur}^{\ast}}(F_{w},T^{\ast})=H^{1}_{ur}(F_{w},T^{\ast})
\end{equation*}
by $(\ref{unramified properties})$. In particular
$H^{1}(F_{w},T^{\ast})/H^{1}_{\mathcal{F}_{ur}^{\ast}}(F_{w},T^{\ast})$
injects into $\mathrm{Hom}(I_{w},T^{\ast})$. Hence
\begin{eqnarray*}
  H^{1}_{\mathcal{F}_{ur}^{+,\ast}}(F,T^{\ast}) &=& \ker(\xymatrix@=1pc{ H^{1}(G_{F,\Sigma},T^{\ast})\ar[r]&
  \bigoplus_{w\in\Sigma}H^{1}(F_{w},T^{\ast})/H^{1}_{\mathcal{F}_{ur}^{+,\ast}}(F_{w},T^{\ast})})\\
   &=&\ker(\xymatrix@=1pc{\mathrm{Hom}(G_{F,\Sigma},T^{\ast})\ar[r]&\bigoplus_{w\in\Sigma_{f}}\mathrm{Hom}(I_{w},T^{\ast})
   }).
\end{eqnarray*}
Using class field theory we obtain the result.
\end{proof}
 Let
\begin{equation*}
H^{1}_{\mathcal{F}_{can}^{+}}(FK_{\infty},T):=\varprojlim_{n}H^{1}_{\mathcal{F}_{can}^{+}}(FK_{n},T),\quad
H^{1}_{\mathcal{F}_{can}^{+,\ast}}(FK_{\infty},T^{\ast}):=\varinjlim_{n}H^{1}_{\mathcal{F}_{can}^{+,\ast}}(FK_{n},T^{\ast}),
\end{equation*}
 where the projective (resp. injective) limit is taken with respect to
 the corestriction (resp. restriction) maps.\vskip 6pt
We now want to study the relation between
$H^{1}_{\mathcal{F}_{can}^{+,\ast}}(K_{\infty},T^{\ast})$ and
$H^{1}_{\mathcal{F}_{can}^{+,\ast}}(L_{\infty},T^{\ast})$.
\begin{deft}\label{definition of Sigma Selmer structure}
We define the Selmer structures $\mathcal{F}_{\Sigma}$ on $T$ by
\begin{itemize}
    \item $\Sigma(\mathcal{F}_{\Sigma})=\Sigma$
     \item $H^{1}_{\mathcal{F}_{\Sigma}}(F_{w},T)= H^{1}(F_{w},T)$
     if $w\in \Sigma$.
\end{itemize}
\end{deft}
 Let
\begin{equation*}
H^{1}_{\mathcal{F}_{\Sigma}^{+}}(FK_{\infty},T):=\varprojlim_{n}H^{1}_{\mathcal{F}_{\Sigma}^{+}}(FK_{n},T),\quad
H^{1}_{\mathcal{F}_{\Sigma}^{+,\ast}}(FK_{\infty},T^{\ast}):=\varinjlim_{n}H^{1}_{\mathcal{F}_{\Sigma}^{+,\ast}}(FK_{n},T^{\ast}),
\end{equation*}
 where the projective (resp. injective) limit is taken with respect to
 the corestriction (resp. restriction) maps.
\begin{lem} Let $\mathcal{G}$ denote the Galois group $\mathrm{Gal}(L_{\infty}/K)$. Then the
 $\mathcal{O}[[\mathcal{G}]]$-modules
$H^{1}_{\mathcal{F}_{can}^{+,\ast}}(L_{\infty},T^{\ast})$ and
$H^{1}_{\mathcal{F}_{\Sigma}^{+,\ast}}(L_{\infty},T^{\ast})$ are
isomorphic.
\end{lem}
\begin{proof} Let  $\Sigma_{2}$ denote the set of $2$-adic places.
Observe that $\mathcal{F}_{\Sigma}^{+,\ast}\leq
\mathcal{F}_{can}^{+,\ast}$, then by $(\ref{exacte sequence of
selmer structures})$ we have an exact sequence
\begin{equation*}
\xymatrix@=2pc{0\ar[r]&
H^{1}_{\mathcal{F}_{\Sigma}^{+,\ast}}(L_{n},T^{\ast})\ar[r]&
 H^{1}_{\mathcal{F}_{can}^{+,\ast}}(L_{n},T^{\ast})\ar[r]&
\displaystyle{\oplus_{w\in\Sigma_{f}-\Sigma_{2}}}H^{1}_{\mathcal{F}_{ur}^{\ast}}(L_{n,w},T^{\ast})}.
\end{equation*}
Passing to direct limit over $n$, the result follows from the proof
of \cite[Proposition 3.5]{AMO1}.
\end{proof}
The following proposition is crucial for our purpose.
 \begin{pro}\label{Proposition pseudo isomorph}
 The $\Lambda$-modules
 $H^{1}_{\mathcal{F}_{can}^{+,\ast}}(K_{\infty},T^{\ast})^{\vee}$
 and $(H^{1}_{\mathcal{F}_{can}^{+,\ast}}(L_{\infty},T^{\ast})^{\vee})_{\mathrm{Gal}(L_{\infty}/K_{\infty})}$
 are pseudo-isomorphic.
 \end{pro}
Before we prove this result we need a preliminary result: For every
 finite Galois extension $F$ of $K$, we have the exact sequence
 \begin{equation}\label{exact sequence of F}
 \xymatrix@=2pc{0\ar[r]&
 H^{1}_{\mathcal{F}^{+,\ast}_{\Sigma}}(F,T^{\ast})^{\vee}\ar[r]&
 H^{2}_{+}(G_{F,\Sigma},T)\ar[r]&\widetilde{\oplus}_{w\in\Sigma_{f}}H^{2}(F_{w},T)\ar[r]&0},
 \end{equation}
 where
 $\widetilde{\oplus}_{w\in\Sigma_{f}}H^{2}(F_{w},T)$
 denotes the kernel of the map
 $ \xymatrix@=2pc{\displaystyle{\oplus_{w\in\Sigma_{f}}}H^{2}(F_{w},T)\ar[r]&
 H^{0}(F,T^{\ast})^{\vee}}.
 $\\ Indeed, by dualizing the exact sequence defining the module
  $H^{1}_{\mathcal{F}^{+,\ast}_{\Sigma}}(F,T^{\ast})$;
\begin{equation*}
\xymatrix@=2pc{0\ar[r]&
H^{1}_{\mathcal{F}_{\Sigma}^{+,\ast}}(F,T^{\ast})\ar[r]&
H^{1}(G_{F,\Sigma},T^{\ast})\ar[r]&\displaystyle{\oplus_{w\in\Sigma_{f}}}
H^{1}(F_{w},T^{\ast})}
\end{equation*}
we obtain the exact sequence
\begin{equation*}
\xymatrix@=2pc{\displaystyle{\oplus_{w\in\Sigma_{f}}}
H^{1}(F_{w},T)\ar[r]&
H^{1}(G_{F,\Sigma},T^{\ast})^{\vee}\ar[r]&H^{1}_{\mathcal{F}_{\Sigma}^{+,\ast}}(F,T^{\ast})^{\vee}\ar[r]&0}.
\end{equation*}
Hence the exact sequence $(\ref{exact sequence of F})$  follows from
 Proposition \ref{exact sequence of finie places}.\vskip 6pt Now we
prove the Proposition \ref{Proposition pseudo isomorph}.
\begin{proof}
Let $n$ be a nonnegative integer and let $\Delta_{n}$ denote the
Galois group $\mathrm{Gal}(L_{n}/K_{n})$. Then the exact sequence
$(\ref{exact sequence of F})$ induces the commutative diagram
\begin{equation*}
\xymatrix@=1pc{  &
(H^{1}_{\mathcal{F}_{\Sigma}^{+,\ast}}(L_{n},T))_{\Delta_{n}}\ar[d]^-{N^{\prime}_{n}}\ar[r]&
H^{2}_{+}(G_{L_{n},\Sigma},T)_{\Delta_{n}}\ar[d]^-{N_{n}}_-{\wr}\ar[r]&
(\widetilde{\oplus}_{w\in\Sigma_{f}}H^{2}(L_{n,w},T))_{\Delta_{n}}\ar[r]\ar[d]^-{N^{\prime\prime}_{n}}&0\\
0\ar[r]&H^{1}_{\mathcal{F}_{\Sigma}^{+,\ast}}(K_{n},T)\ar[r]&
H^{2}_{+}(G_{K_{n},\Sigma},T)\ar[r]&
\widetilde{\oplus}_{w\in\Sigma_{f}}H^{2}(K_{n,w},T)\ar[r]&0}
\end{equation*}
where all vertical maps are induced by the corestriction. The one of
the middle is an isomorphism by Corollary \ref{corolary
coresriction}. By the snake lemma, we obtain
\begin{equation*}
\mathrm{coker}(N^{\prime}_{n})\cong \ker (N^{\prime\prime}_{n})\quad
\mbox{and}\quad \ker(N^{\prime}_{n})\cong \mathrm{coker}(\alpha_{n})
\end{equation*}
where
\begin{equation*}
\xymatrix@=2pc{\alpha_{n}:\;
H_{1}(\Delta_{n},H^{2}_{+}(G_{L_{n},\Sigma},T))\ar[r]&
H_{1}(\Delta_{n},\widetilde{\oplus}_{w\in\Sigma_{f}}H^{2}(L_{n,w},T))}.
\end{equation*}
The orders of the groups
\begin{equation*}
H_{0}(\Delta_{n},\widetilde{\oplus}_{w\in\Sigma_{f}}H^{2}(L_{n,w},T))\quad\mbox{and}\quad
H_{1}(\Delta_{n},\widetilde{\oplus}_{w\in\Sigma_{f}}H^{2}(L_{n,w},T))
\end{equation*}
are bounded independently  of $n$ ($\mathrm{cf.}$\, \cite[Lemma
3.7]{AMO1}). Therefore, the $\Lambda$-modules\\
 $H^{1}_{\mathcal{F}_{can}^{+,\ast}}(K_{\infty},T^{\ast})^{\vee}$
 and $(H^{1}_{\mathcal{F}_{can}^{+,\ast}}(L_{\infty},T^{\ast})^{\vee})_{\mathrm{Gal}(L_{\infty}/K_{\infty})}$
 are pseudo-isomorphic. This finishes the proof.
\end{proof}
For a nonnegative integer $n$, let $A_{n}^{+}$ denote the $2$-part
of the narrow class group of $L_{n}$, and let
\begin{equation*}
A_{\infty}^{+}:=\varprojlim_{n}A_{n}^{+}
\end{equation*}
where the injective limit is taken with respect to the norm maps.
\begin{pro}\label{chi quotient of restrincted class group} If one of the hypotheses $(\mathcal{H}_{2})$ or
$(\mathcal{H}_{3})$ holds then
\begin{equation*}
\mathrm{char}((A^{+}_{\infty})_{\chi})\quad\mbox{divides}\quad
\mathrm{char}(H^{1}_{\mathcal{F}^{+,\ast}_{can}}(K_{\infty},T^{\ast})^{\vee}).
\end{equation*}
\end{pro}
\begin{proof}
 Consider the exact sequence
\begin{equation*}
\xymatrix@=2pc{H^{1}_{\mathcal{F}_{ur}^{\ast}}(L_{n,2},T^{\ast})^{\vee}\ar[r]&
H^{1}_{\mathcal{F}_{ur}^{+,\ast}}(L_{n},T^{\ast})^{\vee}\ar[r]&
H^{1}_{\mathcal{F}_{can}^{+,\ast}}(L_{n},T^{\ast})^{\vee}\ar[r]&0}.
\end{equation*}
Since
\begin{equation*}
H^{1}_{\mathcal{F}_{ur}^{\ast}}(L_{n,2},T^{\ast})\cong \oplus_{w\mid
2}\mathrm{Hom}(G_{w}/I_{w},T^{\ast}),
\end{equation*}
  we obtain
\begin{equation*}
H^{1}_{\mathcal{F}_{ur}^{\ast}}(L_{n,2},T^{\ast})^{\vee}\cong
\oplus_{v\mid
p}\mathcal{O}(\chi^{-1})[\mathrm{Gal}(L_{n}/K)/D_{v}(L_{n}/K)].
\end{equation*}
Passing to the projective limit and taking the
$\Delta$-co-invariants, we get
\begin{equation*}
(\mathcal{O}(\chi^{-1})[\mathcal{G}/D_{v}(L_{\infty}/K)])_{\Delta}\simeq\left\{
                                                                          \begin{array}{ll}
                                                                            \mbox{finite}, & \hbox{if $\chi(D_{v}(L/K))\neq 1$;} \\
                                                                            \mathcal{O}[\mathrm{Gal}(K_{\infty}/K)/D_{v}(K_{\infty}/K)],
                                                                             & \hbox{if  $\chi(D_{v}(L/K))=1$,}
                                                                          \end{array}
                                                                        \right.
\end{equation*}
where $\Delta=\mathrm{Gal}(L_{\infty}/K_{\infty})$. Using
Proposition \ref{Proposition pseudo isomorph} and Lemma $\ref{Lemma
narrow class}$, we obtain
\begin{equation*}
\mathrm{char}((A^{+}_{\infty})_{\chi})\quad\mbox{divides}\quad
\mathcal{J}^{s}\mathrm{char}(H^{1}_{\mathcal{F}^{+,\ast}_{can}}(K_{\infty},T^{\ast})^{\vee})
\end{equation*}
where $\mathcal{J}$ is the augmentation ideal of $\Lambda$ and
$s=\#\{ v\mid 2 \;;\; \chi(\mathrm{Frob}_{v})=1 \}$. Since $L$ is
totally real, the characteristic ideal
$\mathrm{char}((A_{\infty})_{\chi})$ is prime to $\mathcal{J}$, by
Leopoldt conjecture. Then the exact sequence
\begin{equation*}
\xymatrix@=2pc{
\displaystyle{\oplus_{v\mid\infty}}H^{1}_{Iw}(K_{v},T)\ar[r]&
H^{1}_{\mathcal{F}_{ur}^{+,\ast}}(K_{\infty},T^{\ast})^{\vee}\ar[r]&
H^{1}_{\mathcal{F}_{ur}^{\ast}}(K_{\infty},T^{\ast})^{\vee}\ar[r]&0}
\end{equation*}
permits to conclude.
\end{proof}
To obtain some information about the $\Lambda$-structure of
$H^{1}_{\mathcal{F}_{can}^{+}}(K_{\infty},T)$, we need some facts
from universal norms  in $\mathbb{Z}_{2}$-extension \cite{Va 06, Va
09}. Let
\begin{equation*}
    H^{1}_{Iw}(K,.):=\displaystyle{\varprojlim_{n}}H^{1}(G_{K_{n},\Sigma},.)\quad\mbox{and}\quad
    H^{1}_{Iw,+}(K,.):=\displaystyle{\varprojlim_{n}}H^{1}_{+}(G_{K_{n},\Sigma},.).
\end{equation*}
The next proposition is a first step towards Theorem \ref{Theorem
free} below, which claims that the $\Lambda$-modules
$H^{1}_{\mathcal{F}_{can}^{+}}(K_{\infty},T)$ and
$H^{1}_{\mathcal{F}_{can}}(K_{\infty},T)$ are $\Lambda$-free.
\begin{pro}\label{Proposition free}
There are canonical isomorphisms
\begin{enumerate}
    \item $H^{1}_{\mathcal{F}_{can}}(K_{\infty},T)\cong
    H^{1}_{Iw}(K,T)$.
    \item $H^{1}_{\mathcal{F}_{can}^{+}}(K_{\infty},T)\cong
    H^{1}_{Iw,+}(K,T)$.
\end{enumerate}
\end{pro}
\begin{proof}
Let $n$ be a nonnegative integer. By definition we have the exact
sequence
\begin{equation*}
\xymatrix@=1.5pc{ 0\ar[r]& H^{1}_{\mathcal{F}_{can}}(K_{n},T)\ar[r]&
H^{1}(G_{K_{n},\Sigma},T)\ar[r]&
\bigoplus_{w\in\Sigma_{f}-\Sigma_{2}}H^{1}(K_{n,w},T)/H^{1}_{\mathcal{F}_{ur}}(K_{n,w},T)}.
\end{equation*}
Passing to projective limit over $n$, the assertion $(1)$ follows
from the proof of \cite[Proposition 3.5]{AMO1}. In order to obtain
$(2)$,  we have on the one hand the exact sequence
\begin{equation*}
\xymatrix@=2pc{0\ar[r]&H^{1}_{\mathcal{F}_{can}^{+}}(K_{n},T)\ar[r]&H^{1}(G_{K_{n},\Sigma},T)\ar[r]&
\displaystyle{\oplus_{w\mid\infty}}H^{1}(K_{n,w},T)}.
\end{equation*}
On the other hand, by Proposition \ref{Proposition properties of
positive cohomology} we have an exact sequence
\begin{equation}\label{equation 1 proof }
\xymatrix@=1.5pc{\displaystyle{\oplus_{w\mid
\infty}}H^{0}(K_{n,w},T)\ar[r]&
H^{1}_{+}(G_{K_{n},\Sigma},T)\ar[r]&H^{1}(G_{K_{n},\Sigma},T)\ar[r]&
\displaystyle{\oplus_{w\mid \infty}}H^{1}(K_{n,w},T)}.
\end{equation}
Since $T=\mathbb{Z}_{2}(1)\otimes\mathcal{O}(\chi^{-1})$ and $w$ is
a real place, we have $H^{0}(K_{n,w},T)=0$. Hence
\begin{equation*}
H^{1}_{\mathcal{F}_{can}^{+}}(K_{\infty},T)\cong H^{1}_{Iw,+}(K,T).
\end{equation*}
\end{proof}
 We will need  the following isomorphism:
For any Galois extension $F/F^{\prime}$ of number
 fields, $K\subset F^{\prime}\subset F$, the restriction map
 \begin{equation}\label{the restriction isomorphism}
\xymatrix@=1.5pc{ \mathrm{res}: H^{1}(F^{\prime},T)\ar[r]^-{\sim}&
H^{1}(F,T)^{\mathrm{Gal}(F/F^{\prime})}}
 \end{equation}
 induces an isomorphism. Indeed, since $\chi$ has finite order, we can assume that
$\chi(G_{F})=1$. Then
\begin{eqnarray*}
  T^{G_{F}} &=& (\mathbb{Z}_{2}(1)\otimes\mathcal{O}(\chi^{-1}))^{G_{F}} \\
   &=& \mathbb{Z}_{2}(1)^{G_{F}}\otimes\mathcal{O}(\chi^{-1})
\end{eqnarray*}
is trivial. Hence the inflation-restriction exact sequence
\begin{equation*}
\xymatrix@=1.5pc{ 0\ar[r]& H^{1}(F/F^{\prime}, T^{G_{F}})\ar[r]&
H^{1}(F^{\prime},T)\ar[r]&
H^{1}(F,T)^{\mathrm{Gal}(F^{\prime}/F)}\ar[r]& H^{1}(F/F^{\prime},
T^{G_{F}})}
\end{equation*}
gives the isomorphism (\ref{the restriction isomorphism}).\vskip 6pt
For an $\mathcal{O}$-module $M$, let
\begin{itemize}[label= ]
    \item $\mathrm{Tor}_{\mathcal{O}}(M)$: the torsion submodule
    and
    \item
    $\mathrm{Fr}_{\mathcal{O}}(M)=M/\mathrm{Tor}_{\mathcal{O}}(M)$
    : the maximal torsion-free quotient of $M$.
\end{itemize}
\begin{theo}\label{Theorem free}
Suppose that every infinite place of $K$ is completely decomposed in
$L/K$. Then the $\Lambda$-modules
$H^{1}_{\mathcal{F}_{can}^{+}}(K_{\infty},T)$ and
$H^{1}_{\mathcal{F}_{can}}(K_{\infty},T)$ are $\Lambda$-free.
\end{theo}
\begin{proof} By Proposition $\ref{Proposition free}$, it suffices to prove that the $\Lambda$-modules
$H^{1}_{Iw}(K,T)$ and $H^{1}_{Iw,+}(K,T)$ are $\Lambda$-free. For
this we claim that
\begin{enumerate}[label=(\roman*)]
    \item The groups $H^{1}(\Gamma_{n},\mathcal{H}^{1}_{+}(K_{\infty},T))$ and $H^{1}(\Gamma_{n},\mathcal{H}^{1}(K_{\infty},T))$
    are finite, and\\
     $\mathrm{Tor}_{\mathcal{O}}(\mathcal{H}^{1}_{+}(K_{\infty},T))=0$ and $\mathrm{Tor}_{\mathcal{O}}(\mathcal{H}^{1}(K_{\infty},T))=0.$
    \item The groups
    $H^{1}(\Gamma_{n},\mathrm{Fr}_{\mathcal{O}}(\mathcal{H}^{1}_{+}(K_{\infty},T)))$
    and  $H^{1}(\Gamma_{n},\mathrm{Fr}_{\mathcal{O}}(\mathcal{H}^{1}(K_{\infty},T)))$
    are co-finitely generated $\mathcal{O}$-modules,
\end{enumerate}
where $\Gamma_{n}$ denotes the Galois group
$\mathrm{Gal}(K_{\infty}/K_{n})$, and
\begin{equation*}
    \mathcal{H}^{1}(K_{\infty},.):=\displaystyle{\varinjlim_{n}}H^{1}(G_{K_{n},\Sigma},.)\quad\mbox{and}\quad
    \mathcal{H}^{1}_{+}(K_{\infty},.):=\displaystyle{\varinjlim_{n}}H^{1}_{+}(G_{K_{n},\Sigma},.).
\end{equation*}
 Using this claim
Theorem $1.9$ of \cite{Va 06} shows that the $\Lambda$-module
$H^{1}_{Iw}(K,T)$ and $H^{1}_{Iw,+}(K,T)$ are free.\vskip 6pt
\noindent \textit{Proof of the claim}: On the one hand the
Hochschild-Serre spectral sequence (see Proposition \ref{Proposition
properties of positive cohomology})
\begin{equation*}
H^{p}(\Gamma_{n},\mathcal{H}_{+}^{q}(K_{\infty},T))\Longrightarrow
H_{+}^{p+q}(G_{K_{n},\Sigma},T)
\end{equation*}
 induces  the exact sequence
\begin{equation*}
\xymatrix@=1pc{H^{1}(\Gamma_{n},\mathcal{H}^{1}_{+}(K_{\infty},T))\ar@{^{(}->}[r]&
H^{2}_{+}(G_{K_{n},\Sigma},T)\ar[r]&
\mathcal{H}^{2}_{+}(K_{\infty},T)^{\Gamma_{n}}\ar[r]&H^{2}(\Gamma_{n},\mathcal{H}^{1}_{+}(K_{\infty},T))}.
\end{equation*}
Since $H^{2}_{+}(G_{K_{n},\Sigma},T)$ is a finitely generated
$\mathcal{O}$-module and
$H^{1}(\Gamma_{n},\mathcal{H}^{1}_{+}(K_{\infty},T))$ is
$\mathcal{O}$-torsion, the module
$H^{1}(\Gamma_{n},\mathcal{H}^{1}_{+}(K_{\infty},T))$ is finite.
 On the other
hand the Hochschild-Serre spectral sequence
\begin{equation*}
H^{p}(\Gamma_{n},\mathcal{H}^{q}(K_{\infty},T))\Longrightarrow
H^{p+q}(G_{K_{n},\Sigma},T)
\end{equation*}
shows that
\begin{equation*}
\xymatrix@=2pc{
H^{1}(\Gamma_{n},\mathcal{H}^{1}(K_{\infty},T))\ar@{^{(}->}[r]&
H^{2}(G_{K_{n},\Sigma},T)}.
\end{equation*}
Hence $H^{1}(\Gamma_{n},\mathcal{H}^{1}(K_{\infty},T))$ is finite.
Since  $L_{n}$ is totally real, by the isomorphism $(\ref{the
restriction isomorphism})$, we get
$\mathrm{Tor}_{\mathcal{O}}(H^{1}(G_{K_{n},\Sigma},T))=0$.
Therefore, the exact sequence $(\ref{equation 1 proof })$ shows that
$\mathrm{Tor}_{\mathcal{O}}(H^{1}_{+}(G_{K_{n},\Sigma},T))=0$, hence
\begin{equation*}
\mathrm{Tor}_{\mathcal{O}}(\mathcal{H}^{1}_{+}(K_{\infty},T))=\mathrm{Tor}_{\mathcal{O}}(\mathcal{H}^{1}(K_{\infty},T))=0.
\end{equation*}
This proves the assertion $(\textit{i})$. The assertion
$(\textit{ii})$ is a direct consequence of $(\textit{i})$.
\end{proof}
\begin{coro}\label{Corollary rank of positive structure}
Suppose that every infinite place of $K$ is completely decomposed in
$L/K$. Then
\begin{equation*}
\mathrm{rank}_{\Lambda}(H^{1}_{\mathcal{F}_{can}^{+}}(K_{\infty},T))=
\mathrm{rank}_{\Lambda}(H^{1}_{\mathcal{F}_{can}}(K_{\infty},T))=[K:\mathbb{Q}].
\end{equation*}
\end{coro}
\begin{proof} Since $\mathcal{F}_{cn}^{+}\leq \mathcal{F}_{cn}$, by
$(\ref{exacte sequence of selmer structures})$ we have an exact
sequence
\begin{equation*}
\xymatrix@=2pc{ 0\ar[r]&
H^{1}_{\mathcal{F}_{can}^{+}}(K_{\infty},T)\ar[r]&
H^{1}_{\mathcal{F}_{can}}(K_{\infty},T)\ar[r]&
\displaystyle{\varprojlim_{n}}(\displaystyle{\oplus_{w\mid\infty}}H^{1}(K_{n,w},T))}.
\end{equation*}
Then
\begin{equation*}
\mathrm{rank}_{\Lambda}(H^{1}_{\mathcal{F}_{can}^{+}}(K_{\infty},T))=
\mathrm{rank}_{\Lambda}(H^{1}_{\mathcal{F}_{can}}(K_{\infty},T)).
\end{equation*}
Using the fact that $H^{1}(G_{L_{n},\Sigma},\mathbb{Z}_{2}(1))\cong
U_{\Sigma}(L_{n})\otimes\mathbb{Z}_{2}$, where $U_{\Sigma}(L_{n})$
denotes the $\Sigma$-units of $L_{n}$. Dirichlet's unit theorem and
the isomorphism $(\ref{the restriction isomorphism})$ show that
\begin{equation*}
\mathrm{rank}_{\mathcal{O}}(H^{1}(G_{K_{n},\Sigma},T))=r2^{n}+t,
\end{equation*}
where $r=[K:\mathbb{Q}]$ and $t$ is an integer independent of $n$.
Then
$\mathrm{rank}_{\Lambda}(H^{1}_{\mathcal{F}_{can}}(K_{\infty},T))=r$.
This proves the corollary .
\end{proof}
\section{\bf Proof of Theorem \ref{Mazigh 2}}
 We will take the notations and the conventions of \cite{Ma 16}. In
particular, the construction of the group of Rubin-Stark units
\cite[Definition 4.5]{Ma 16} goes on the  same lines. \vskip 5pt For
a nonnegative integer $n$, the product of all distinct non $2$-adic
prime ideals dividing the finite part of the conductor of $L_{n}/K$
is denoted by $\widetilde{\mathfrak{h}}$, which does not depend on
$n$. For any ideal $\mathfrak{g}\mid \widetilde{\mathfrak{h}}$, the
maximal subextension of $L_{n}$ whose conductor is prime to
$\widetilde{\mathfrak{h}}\mathfrak{g}^{-1}$ is denoted by
$L_{n,\mathfrak{g}}$. Let us fix a finite set $S$ of places
containing all infinite places, and at least one finite place, but
does not contain any $2$-adic prime of $K$, and a second finite,
nonempty set $\mathcal{T}$ of places of $K$, disjoint from $S$ and
does not contain any $2$-adic prime of $K$. Let
$S_{L_{n,\mathfrak{g}}}=S\cup \mathrm{Ram}(L_{n,\mathfrak{g}}/K)$,
where $\mathrm{Ram}(L_{n,\mathfrak{g}}/K)$ denotes the set of
ramified primes in $L_{n,\mathfrak{g}}/K$. Since
$L_{n,\mathfrak{g}}$ is a totally real field, the hypotheses
$2.1.1$-$2.1.5$ in \cite[Hypotheses 2.1]{Rubin96} on
$S_{L_{n,\mathfrak{g}}}$, $\mathcal{T}$ and $r$  are satisfied.\\
 Let $\mathcal{E}_{n}$ (resp.
$\mathcal{E}_{n}^{+}$) denote  the group of units (resp. totally
positive units) of $L_{n}$. Following \cite{Ma 16}, we define
\begin{deft}\label{Definition of Rubin Strak module}
Let $n$ be a nonnegative integer. We denote by $\mathrm{St}_{n}^{+}$
the $\mathbb{Z}[\mathrm{Gal}(L_{n}/K)]$-module generated by the
inverse images of $\varepsilon_{n,\mathfrak{g},\mathcal{T}}$ under
the map $ \xymatrix@=2pc{ \bigwedge^{r}\mathcal{E}_{n}^{+}\ar[r]&
 \mathbb{Q}\otimes\bigwedge^{r}\mathcal{E}_{n}}$
 for all $\mathfrak{g}\mid \widetilde{\mathfrak{h}}$,
 where $\varepsilon_{n,\mathfrak{g},\mathcal{T}}$ is the Rubin-Stark element of
 the Rubin-Stark conjecture $\mathbf{RS}(L_{n,\mathfrak{g}}/K,S_{L_{n,\mathfrak{g}}},\mathcal{T},r)$
 \cite[Conjecture $\mathrm{B}^\prime$]{Rubin96}.
\end{deft}
 Recall that for any number field $F$, Kummer theory gives
a canonical isomorphism
\begin{equation*}
H^{1}(F,\mathbb{Z}_{2}(1))\cong F^{\times,\wedge}:=\varprojlim
F^{\times}/(F^{\times})^{2^{n}}.
\end{equation*}
Since  $\chi(G_{L_{n}})=1$ for every $n\geq0$,
\begin{equation*}
H^{1}(L_{n},\mathbb{Z}_{2}(1))\otimes \mathcal{O}(\chi^{-1})\cong
H^{1}(L_{n},\mathbb{Z}_{2}(1)\otimes\mathcal{O}(\chi^{-1})).
\end{equation*}
 Therefore
 \begin{equation}\label{ismophism Lng}
L_{n}^{\times,\wedge}\otimes \mathcal{O}(\chi^{-1})\cong
H^{1}(L_{n},\mathbb{Z}_{2}(1)\otimes\mathcal{O}(\chi^{-1})).
 \end{equation}
For simplicity of notation, we let $\varepsilon_{n}$ stand for the
Rubin-Stark element
$\varepsilon_{n,\widetilde{\mathfrak{h}},\mathcal{T}}$ for
$\mathbf{RS}(L_{n}/K,S_{L_{n}},\mathcal{T},r)$. It is well known
that $\varepsilon_{n}$ can be uniquely written as
$\varepsilon_{1}\wedge\cdots\wedge \varepsilon_{r}$, with
$\varepsilon_{i}\in \mathbb{Q}\otimes L_{n}^{\times}$.
 Let
\begin{equation}\label{definition of chi element of p copletion}
\varepsilon_{n,\chi}:=\widehat{\varepsilon_{1}}\otimes1_{\chi^{-1}}\wedge\cdots\wedge
\widehat{\varepsilon_{r}}\otimes 1_{\chi^{-1}},
\end{equation}
where $\widehat{\varepsilon_{i}}$ is the image of $\varepsilon_{i}$
by the natural map $ \xymatrix@=1.5pc{ \mathbb{Q}\otimes
L_{n}^{\times}\ar[r]& \mathbb{Q}_{2}\otimes_{\mathbb{Z}_{2}}
L_{n}^{\times, \wedge}}.$
 Then, under the isomorphism $(\ref{ismophism Lng})$, we can
view each
\begin{equation*}
\mbox{ $\varepsilon_{n,\chi}$ as an element of
$\mathbb{Q}_{2}\otimes\displaystyle{\bigwedge^{r}}H^{1}(L_{n},\mathbb{Z}_{2}(1)\otimes\mathcal{O}(\chi^{-1})).$
}
\end{equation*}
 For every $n\geq 0$, we define
\begin{equation}\label{definition of Engr}
 c_{n}= \mathrm{cor}_{L_{n+1},K_{n}}^{(r)} (\varepsilon_{n+1,\chi})
\end{equation}
where $\mathrm{cor}_{L_{n+1},K_{n}}^{(r)}$ is the map
\begin{equation*}
\xymatrix@=2pc{ \mathbb{Q}_{2}\otimes
\displaystyle{\bigwedge^{r}}H^{1}(L_{n+1},T)\ar[r]&
\mathbb{Q}_{2}\otimes \displaystyle{\bigwedge^{r}}H^{1}(K_{n},T)}
\end{equation*}
induced by the corestrection map
\begin{equation*}
\xymatrix@=2pc{ \mathrm{cor}_{L_{n+1},K_{n}}:
H^{1}(L_{n+1},T)\ar[r]& H^{1}(K_{n},T)}.
\end{equation*}
Let $\iota$ denote the composite of the natural maps:
\begin{equation*}
 \xymatrix@=1.5pc{
 \displaystyle{\bigwedge^{r}\varprojlim_{n}}H^{1}_{+}(K_{n},T)\ar[r]&\displaystyle{
 \varprojlim_{n}\bigwedge^{r}}H^{1}_{+}(K_{n},T)\ar[r]&\displaystyle{\varprojlim_{n}
 (\mathbb{Q}_{2}\otimes_{\mathbb{Z}_{2}}\bigwedge^{r}}H^{1}(K_{n},T))}
 \end{equation*}
 and let
 \begin{equation*}
 \xymatrix@=2pc{
 \tau:\,\displaystyle{\bigwedge^{r}\varprojlim_{n}}H^{1}_{+}(K_{n},T)\ar[r]&
  \displaystyle{\bigwedge^{r}\varprojlim_{n}}H^{1}(K_{n},T)}.
 \end{equation*}
 By Corollary \ref{Corollary rank of positive structure}, it is
 clear that $\tau$ is injective.\\
 Let $c_{\infty}:=\{c_{n}\}_{n\geq
0}\in\displaystyle{\varprojlim_{n}}(\mathbb{Q}_{2}\otimes_{\mathbb{Z}_{2}}\bigwedge^{r}H^{1}(K_{n},T))$,
where $c_{n}$ is defined in $(\ref{definition of Engr})$. The
collection $\{c_{n}\}_{n\geq 0}$ gives rise to an Euler system for
the $2$-adic representation
$T=\mathbb{Z}_{2}(1)\otimes\mathcal{O}(\chi^{-1})$, in the sense of
\cite[Definition 2.1.1]{Rubin00} (see $\mathrm{e.g.}$
\cite[Proposition 4.7]{Ma 16}). Recall that we can associate a
Kolyvagin derivative class to any Euler system for any $2$-adic
representation \cite[\S 4.4]{Rubin00}. In the sense of
\cite[Definition 3.1.3]{MR04}, this turn out to construct  a
Kolyvagin system of the canonical Selmer structure
$\mathcal{F}_{can}$ \cite[Theorem 3.2.4]{MR04}.
\begin{lem}\label{lemme mazigh}
Let $\eta$ denote the $\mathrm{lcm}$ of the $2$-adic numbers
$1-\chi(\mathrm{Frob}_{\mathfrak{p}})$, where $\mathfrak{p}$ run
through the set of $2$-adic place of $K$. Let $\mathbf{c}$ be an
element in $\iota^{-1}(\eta.c_{\infty})$. Under the hypothesis
$(\mathcal{H}_{3})$,
\begin{equation*}
\mathrm{char}(H^{1}_{\mathcal{F}_{can}^{\ast}}(K_{\infty},T^{\ast})^{\vee})\quad\mbox{divides}\quad
\mathrm{char}((\bigwedge^{r}H^{1}_{\mathcal{F}_{can}}(K_{\infty},T))/\Lambda.\tau(\mathbf{c})).
\end{equation*}
\end{lem}
\begin{proof} The proof is identical to the proof of Theorem $6.3$ of \cite{Ma
16} line by line. To obtain Theorem \cite[Theorem 6.3]{Ma 16} we
proved a variant of Rubin's theorem (\cite[Theorem 2.3.3]{Rubin00}),
loc.cit \cite[Theorem 6.1]{Ma 16} by constructing an ad-hoc Selmer
structure \cite[Definition 5.6]{Ma 16} and an associated Kolyvagin
system \cite[Lemma 5.13]{Ma 16}. The construction uses the structure
of $H^{1}_{Iw}(K(\mathfrak{r})_{v},T)$ \cite[Theorem 5.1]{Ma 16},
deduced from a result of Greither \cite[Theorem 2.2]{Greither96}.
Since Greither's result is also available for $p=2$
\cite[Proposition 2.10]{Greither92}, the strategy used to obtain
\cite[Theorem 6.3]{Ma 16} is also applicable for Lemma \ref{lemme
mazigh}.
\end{proof}
For each place $v$ of $K$, let
\begin{equation*}
H^{1}_{Iw}(K_{v},T)=\displaystyle{\varprojlim_{n}}(\displaystyle{\oplus_{w\mid
v}}H^{1}(K_{n,w},T)).
\end{equation*}
By a standard argument (see \cite[Lemma 5.3.1]{MR04}) we have
\begin{equation*}
H^{1}_{Iw}(K_{v},T)\cong H^{1}(K_{v},T\otimes\Lambda).
\end{equation*}
Hence the Proposition $4.2.3$ of \cite{Nekovar06} shows that
$H^{1}_{Iw}(K_{v},T)$ is a finitely generated
$\Lambda$-module. \\
The following proposition is the key to the proof of our main
theorem.
\begin{pro}\label{proposition mazigh 1} With the assumptions of
Lemma \ref{lemme mazigh}, we have
\begin{equation*}
\mathrm{char}(H^{1}_{\mathcal{F}_{can}^{+,\ast}}(K_{\infty},T^{\ast})^{\vee})\quad\mbox{divides}\quad
\mathrm{char}((\bigwedge^{r}H^{1}_{\mathcal{F}_{can}^{+}}(K_{\infty},T))/\Lambda.\mathbf{c}).\mathrm{char}
(\displaystyle{\oplus_{v\mid\infty}}H^{1}_{Iw}(K_{v},T)).
\end{equation*}
\end{pro}
\begin{proof}
Since $\mathcal{F}_{can}^{+}\leq \mathcal{F}_{can}$, we have an
exact sequence
\begin{equation}\label{exact sequence proof of pro}
\xymatrix@=1.5pc{H^{1}_{\mathcal{F}_{can}^{+}}(K_{\infty},T)\ar@{^{(}->}[r]&
H^{1}_{\mathcal{F}_{can}}(K_{\infty},T)\ar[r]&
\displaystyle{\oplus_{v\mid\infty}}H^{1}_{Iw}(K_{v},T)\ar[r]&
H^{1}_{\mathcal{F}_{can}^{+,\ast}}(K_{\infty},T^{\ast})^{\vee}\ar@{->>}[r]&H^{1}_{\mathcal{F}_{can}^{\ast}}
(K_{\infty},T^{\ast})^{\vee} }.
\end{equation}
Corollary \ref{Corollary rank of positive structure} shows that the
$\Lambda$-modules $H^{1}_{\mathcal{F}_{can}^{+}}(K_{\infty},T)$ and
$H^{1}_{\mathcal{F}_{can}}(K_{\infty},T)$ are $\Lambda$-free of rank
$r=[K:\mathbb{Q}]$, and therefore the injection $\xymatrix@=2pc{
H^{1}_{\mathcal{F}_{can}^{+}}(K_{\infty},T)\ar@{^{(}->}[r]^-{\beta}&H^{1}_{\mathcal{F}_{can}^{+}}(K_{\infty},T)
}$ induces an exact sequence:
\begin{equation*}
\xymatrix@=2pc{ 0\ar[r]&
(\displaystyle{\bigwedge^{r}}H^{1}_{\mathcal{F}_{can}^{+}}(K_{\infty},T))/\Lambda.\mathbf{c}\ar[r]&
(\displaystyle{\bigwedge^{r}}H^{1}_{\mathcal{F}_{can}}(K_{\infty},T))/\Lambda.\mathbf{c}\ar@{->>}[r]&
\mathrm{coker}(\beta^{(r)})},
\end{equation*}
where $\beta^{(r)}$ denotes the map induced on the $r$-th exterior
power. Using the fact that
\begin{equation*}
\mathrm{char}(\mathrm{coker}(\beta))=\mathrm{char}(\mathrm{coker}(\beta^{(r)}))
\end{equation*}
($\mathrm{cf.}$\,\cite[page 258]{Bourbaki}), we get
 \begin{align*}
  & \mathrm{char}((\bigwedge^{r}H^{1}_{\mathcal{F}_{can}^{+}}(K_{\infty},T))/\Lambda.\mathbf{c}).\mathrm{char}
(\displaystyle{\oplus_{v\mid\infty}}H^{1}_{Iw}(K_{v},T)).
\mathrm{char}(H^{1}_{\mathcal{F}_{can}^{\ast}}(K_{\infty},T^{\ast})^{\vee})&
\\
=\; &
\mathrm{char}((\bigwedge^{r}H^{1}_{\mathcal{F}_{can}}(K_{\infty},T))/\Lambda.\mathbf{c}).
\mathrm{char}(H^{1}_{\mathcal{F}_{can}^{+,\ast}}(K_{\infty},T^{\ast})^{\vee}).
&
 \end{align*}
 Lemma \ref{lemme mazigh} permits to conclude.
\end{proof}
Let $n$ be a nonnegative integer, we write $A^{+}_{n}$ for the
$2$-part of the narrow  class group of $L_{n}$,
$\mathcal{E}_{n}^{\prime}$ for the $2$-units of $L_{n}$ and
$\mathcal{E}_{n}^{\prime,+}$ for the totally positive $2$-units of
$L_{n}$. Let
\begin{equation*}
A^{+}_{\infty}:=\displaystyle{\varprojlim_{n}}A_{n}^{+},\quad\quad
\displaystyle{\widehat{\mathcal{E}_{\infty}^{\prime}}}:=\displaystyle{\varprojlim_{n}}\displaystyle{\widehat{\mathcal{E}_{n}^{\prime}}},
\quad\quad
\displaystyle{\widehat{\mathcal{E}_{\infty}^{\prime,+}}}:=\displaystyle{\varprojlim_{n}}\displaystyle{\widehat{\mathcal{E}_{n}^{\prime,+}}},
\end{equation*}
where all inverse limits are taken with respect to norm maps. It is
well known that
\begin{equation*}
\displaystyle{\varprojlim_{n}}H^{1}(G_{L_{n},\Sigma},\mathbb{Z}_{2}(1))\cong
\displaystyle{\widehat{\mathcal{E}_{\infty}^{\prime}}}.
\end{equation*}
Since $L_{n}$ is  a totally real field, Proposition \ref{Proposition
properties of positive cohomology} leads an exact sequence
\begin{equation*}
\xymatrix@=2pc{ 0\ar[r]&
H^{1}_{+}(G_{L_{n},\Sigma},\mathbb{Z}_{2}(1))\ar[r]&
H^{1}(G_{L_{n},\Sigma},\mathbb{Z}_{2}(1))\ar[r]&
\displaystyle{\oplus_{w\mid\infty}}H^{1}(L_{n,w},\mathbb{Z}_{2}(1)))}.
\end{equation*}
 Hence
\begin{equation}\label{Isomorphism totally positive units}
\displaystyle{\varprojlim_{n}}H^{1}_{+}(G_{L_{n},\Sigma},\mathbb{Z}_{2}(1))\cong
\displaystyle{\widehat{\mathcal{E}_{\infty}^{\prime,+}}}.
\end{equation}
 Recall that $St_{n}^{+}$ denotes the
  $\mathbb{Z}[\mathrm{Gal}(L_{n}/K)]$-module constructed by the
  Rubin-Stark elements (see Definition \ref{Definition of Rubin Strak
  module}). Recall also that
\begin{equation*}
c_{n}=\mathrm{cor}^{(r)}_{L_{n+1},K_{n}}(\varepsilon_{n+1,\chi})
\end{equation*}
denotes the element defined in $(\ref{definition of Engr})$. Let
$St^{+}_{\infty}:=\varprojlim_{n}St_{n}^{+}$ and let
$\varepsilon_{\infty,\chi}:=\{\varepsilon_{n,\chi}\}_{n\geq 1}$.
Since for $n\geq 1$, $
c_{n}=\mathrm{cor}^{(r)}_{L_{n},K_{n}}(\varepsilon_{n,\chi}) $, it
follows that
\begin{eqnarray*}
  \mathrm{res}^{(r)}_{K_{n},L_{n}}(c_{n}) &=& \mathrm{res}^{(r)}_{K_{n},L_{n}}(\mathrm{cor}^{(r)}_{L_{n},K_{n}}(\varepsilon_{n,\chi})) \\
   &=& |\Delta|^{r-1}N_{\Delta}(\varepsilon_{n,\chi}).
\end{eqnarray*}
 Therefore, using the fact that the restriction map
\begin{equation*}
\xymatrix@=2pc{\mathrm{res}_{K_{n},L_{n}} : H^{1}(K_{n},T)\ar[r]&
H^{1}(L_{n},T)^{\mathrm{Gal}(L_{n}/K_{n})}}
\end{equation*}
is an isomorphism by $(\ref{the restriction isomorphism})$, we
obtain
\begin{equation*}
|\Delta|^{r-1}N_{\Delta}((\displaystyle{\widehat{St_{\infty}^{+}}})_{\chi})=
\Lambda \mathbf{c},
\end{equation*}
where $\mathbf{c}$ is the inverse image of
$|\Delta|^{r-1}N_{\Delta}(\varepsilon_{\infty,\chi})$ under the
composite map:
\begin{equation*}
\xymatrix@=2pc{\bigwedge^{r}\varprojlim_{n}H^{1}_{+}(K_{n},T)\ar[r]&
\varprojlim_{n}\bigwedge^{r}H^{1}_{+}(K_{n},T)\ar[r]&
\varprojlim_{n}(\mathbb{Q}_{2}\otimes_{\mathbb{Z}_{2}}\bigwedge^{r}H^{1}(K_{n},T))}.
\end{equation*}
Using Proposition $\ref{Proposition free}$ and the isomorphisms
$(\ref{the restriction isomorphism})$ and $(\ref{Isomorphism totally
positive units})$, we get
\begin{equation*}
H^{1}_{\mathcal{F}_{can}^{+}}(K_{\infty},T)\cong
(\displaystyle{\widehat{\mathcal{E}_{\infty}^{\prime,+}}}\otimes\mathcal{O}(\chi^{-1}))^{\mathrm{Gal}(L_{\infty}/K_{\infty})}.
\end{equation*}
 \noindent \textbf{Proof Theorem \ref{Mazigh 2}.}
 Consider the commutative exact diagram
\begin{equation*}
\xymatrix@=1.5pc{ &
(\widehat{\mathrm{St}_{\infty}^{+}})_{\chi}\ar[r]\ar@{->>}[d]^-{|\Delta|^{r-1}N_{\Delta}}&
\displaystyle{\bigwedge^{r}}(\widehat{\mathcal{E}_{\infty}^{+}})_{\chi}\ar@{->>}[r]\ar[d]^-{N^{(r)}_{\Delta}}&
\big(\displaystyle{\bigwedge^{r}}\widehat{\mathcal{E}_{\infty}^{+}}/\widehat{\mathrm{St}_{\infty}^{+}}\big)_{\chi}\ar[d]\\
0\ar[r]&|\Delta|^{r-1}N_{\Delta}(\mathbf{c})\ar[r]&
\displaystyle{\bigwedge^{r}}(\widehat{\mathcal{E}_{\infty}^{+}})^{\chi}\ar@{->>}[r]\ar@{->>}[d]&
\displaystyle{\bigwedge^{r}}(\widehat{\mathcal{E}_{\infty}^{+}})^{\chi}/|\Delta|^{r-1}N_{\Delta}(\mathbf{c})\\
  &  & \mathrm{coker}(N^{(r)}_{\Delta})& }
\end{equation*}
where
$(\widehat{\mathcal{E}_{\infty}^{+}})^{\chi}=(\widehat{\mathcal{E}_{\infty}^{+}}\otimes_{\mathbb{Z}_{2}}
\mathcal{O}(\chi^{-1}))^{\Delta}$. Then
 \begin{equation*}
 \mathrm{char}(\displaystyle{\bigwedge^{r}}(\widehat{\mathcal{E}_{\infty}})^{\chi}/
 |\Delta|^{r-1}N_{\Delta}(\mathbf{c}))
\quad\mbox{divides}\quad
\mathrm{char}\bigg(\big(\displaystyle{\bigwedge^{r}}\widehat{\mathcal{E}_{\infty}}/\widehat{\mathrm{St}_{\infty}}\big)_{\chi}\bigg).
\mathrm{char}(\mathrm{coker}(N^{(r)}_{\Delta}))
\end{equation*}
Since $\chi(D_{v}(L/K))\neq 1$ for any $2$-adic prime of $K$, we get
\begin{equation*}
(\widehat{\mathcal{E}^{\prime,+}_{\infty}}\otimes_{\mathbb{Z}_{2}}\mathcal{O}(\chi^{-1}))^{\Delta}
\cong
(\widehat{\mathcal{E}_{\infty}^{+}}\otimes_{\mathbb{Z}_{2}}\mathcal{O}(\chi^{-1}))^{\Delta}.
\end{equation*}
Hence by Proposition \ref{proposition mazigh 1} and Proposition
\ref{chi quotient of restrincted class group}, we get
\begin{equation*}
\mathrm{char}((A^{+}_{\infty})_{\chi})\quad \mbox{divides}\quad
\lambda. \mathrm{char}
\bigg(\big(\displaystyle{\bigwedge^{r}}\widehat{\mathcal{E}_{\infty}^{+}}/\widehat{\mathrm{St}_{\infty}^{+}}\big)_{\chi}\bigg).
\end{equation*}
where
\begin{equation}\label{lamda factur}
\lambda=\eta.\mathrm{char}(\mathrm{coker}(N^{(r)}_{\Delta})).\mathrm{char}
(\displaystyle{\oplus_{v\mid\infty}}(H^{1}_{Iw}(K_{v},T)))
\end{equation}
and $\eta$ is the $\mathrm{lcm}$ of the $2$-adic numbers
$1-\chi(\mathrm{Frob}_{\mathfrak{p}})$, where $\mathfrak{p}$ run
through the set of $2$-adic place of $K$.
\begin{rem}
The cokernel of the morphism
\begin{equation*}
\xymatrix@=2pc{(\widehat{\mathcal{E}_{\infty}^{+}})_{\chi}\ar[r]^-{N_{\Delta}}&
(\widehat{\mathcal{E}_{\infty}^{+}})^{\chi}}
\end{equation*}
is  isomorphic to $\widehat{H}^{0}(\Delta,
\widehat{\mathcal{E}_{\infty}^{+}}\otimes_{\mathbb{Z}_{2}}\mathcal{O}(\chi^{-1}))$,
where $\widehat{H}^{0}(.,.)$ denotes the modified Tate cohomology
group. The module  $\mathrm{coker}(N^{(r)}_{\Delta})$ is then a
finitely generated torsion $\Lambda$-module, annihilated by
$|\Delta|$. Hence the characteristic ideal
$\mathrm{char}(\mathrm{coker}(N^{(r)}_{\Delta}))$ is a power of $2$.
By a standard argument (see $\mathrm{e.g.}$ \cite[Lemma
5.3.1]{MR04}) we have
\begin{equation*}
\displaystyle{\oplus_{v\mid\infty}}H^{1}_{Iw}(K_{v},T)\cong
\displaystyle{\oplus_{v\mid\infty}}H^{1}(K_{v},T\otimes\Lambda).
\end{equation*}
 Moreover, as the absolute Galois group of the field
 $K_{v}=\mathbb{R}$ is  cyclic  of order $2$, using the cohomology of cyclic
groups, we show that
\begin{equation*}
H^{1}(K_{v},T\otimes\Lambda)\cong
\mathbb{Z}/2\mathbb{Z}\otimes_{\mathbb{Z}}
(T\otimes_{\mathbb{Z}_{2}}\Lambda).
\end{equation*}
Therefore
\begin{equation*}
\lambda=2^{r}.\eta.\mathrm{char}(\mathrm{coker}(N^{(r)}_{\Delta}))
\end{equation*}
is a power of $2$, where $r=[K:\mathbb{Q}]$.
\end{rem}


\begin{thebibliography}{99}
     \setlength{\parskip}{0cm}
 \bibitem[AMO]{AMO1}{\bf J. Assim,\, Y. Mazigh,\, H. Oukhaba.} {\it  Th\'{e}orie d'Iwasawa des unit\'{e}s de Stark et groupe de
  classes.}  Int. J. Number Theory 13 (2017), no. 5, 1165-1190.
  \bibitem[As-Mo]{Assim mova} { \bf J. Assim, A. Movahhedi.} {\it
Galois codescent for motivic tame kernels}. preprint.
  \bibitem[Bo]{Bourbaki} {\bf N. Bourbaki.} {\it Alg\`{e}bre commutative, chapitre $7$, Diviseurs}, Hermann, 1965
     \bibitem[B\"{u} 09]{Kazim109}{\bf K. B\"{u}y\"{u}kboduk.} {\it Stark units and the main conjectures for totally real fields}.
      Compos. Math. 145 (2009), no. 5, 1163-1195.
 \bibitem[CKPS]{CKPS}{\bf T.Chinburg,\, M.Kolster,\, V.Pappas, and  V.Snaith.} {\it Galois
structure of $K$-groups of rings of integers}. $K$-Theory 14 (1998),
319-369.
\bibitem[Gr 92]{Greither92}{\bf  C. Greither.} {\it Class groups of abelian fields, and the main conjecture.} Ann. Inst. Fourier (Grenoble) 42
       (1992), no. 3, 449-499.
\bibitem[Gr 96]{Greither96}{\bf C. Greither.} {\it On Chinburg's second conjecture for abelian
 fields}. J. Reine Angew. Math. 479 (1996), 1-37.
 \bibitem[Ka 93]{Ka 93} {\bf B. Kahn.} {\it Descente galoisienne et $K_{2}$ des corps de nombres}. $\mathrm{K}$-Theory 7 (1993),
55-100.
\bibitem[Ma 16]{Ma 16} {\bf Y. Mazigh.} {\it Iwasawa theory of Rubin-Stark units and class
groups.} Manuscripta Math. 153 (2017), no. 3-4, 403-430.
 \bibitem[MR 04]{MR04}{\bf B. Mazur,\, K. Rubin.} {\it Kolyvagin systems}. Mem. Amer. Math. Soc., 168(799):\textrm{viii}+96, 2004.
\bibitem[MR 16 ]{MR 16}{\bf B. Mazur,\,K. Rubin.} {\it Controlling Selmer groups in the higher core rank
case}. J. Théor. Nombres Bordeaux 28 (2016), no. 1, 145-183.
\bibitem[Mi 86]{Milne}{\bf J. Milne.} {\it Arithmetic duality theorems.}
Acad. Press, Boston, 1986.
\bibitem[Ne 06]{Nekovar06}{\bf J. Nekov\'{a}r.} {\it Selmer complexes}. Ast\'{e}risque 310 (2006).
\bibitem[Ou 12]{Oukhaba 12} {\bf H. Oukhaba.} {\it On Iwasawa theory of elliptic units and $2$-ideal class
groups.} J. Ramanujan Math. Soc. 27 (2012), no. 3, 255-27.
\bibitem[PR 92]{PR92}{\bf B. Perrin-Riou.} {\it Th\'{e}orie d'Iwasawa et hauteurs p-adiques}. Invent. Math. 109 (1992) 137-185.
\bibitem[Ru 96]{Rubin96}{\bf  K. Rubin.} {\it  A Stark conjecture "over $\mathbb{Z}$" for abelian $L$-functions with multiple
zeros}. Ann. Inst. Fourier (Grenoble), 46(1):33-62, 1996.
 \bibitem[Ru 00]{Rubin00}  {\bf K. Rubin.} {\it  Euler systems}. Annals of Mathematics Studies, 147.
       Hermann Weyl Lectures. The Institute for Advanced Study. Princeton University Press, Princeton,
 \bibitem[Ta 84]{Tate84} {\bf J. Tate.} {\it Les conjectures de Stark sur les fonctions L d'Artin en s=0}.
Birkh\"{a}user Boston Inc, 1984. Lecture notes edited by Dominique
Bernardi and Norbert Schappacher.
\bibitem[Va 06]{Va 06} {\bf D. Vauclair.} {\it Sur les normes universelles et la structure de
certains modules d'Iwasawa}. (2006) No publishes. Available at
\url{http://www.math.unicaen.fr/~vauclair/}
\bibitem[Va 09]{Va 09} {\bf D. Vauclair.} {\it Sur la dualité et la descente
d'Iwasawa.} , Ann. Inst. Fourier 59 (2009), no 2, 691-767
\end{thebibliography}
\end{document}